\documentclass[12pt,a4paper]{article}
\usepackage[utf8]{inputenc}
\usepackage[UKenglish]{isodate}
\cleanlookdateon
\usepackage{amssymb}
\usepackage{amsmath}
\usepackage{amsthm}
\usepackage{mathtools}
\usepackage{enumitem}
\usepackage[usenames,dvipsnames]{xcolor}
\usepackage[T1]{fontenc}
\usepackage{parskip}
\usepackage[sort,compress,numbers]{natbib}
\usepackage[hyphens]{url}
\usepackage{hyperref}
\usepackage[hyphenbreaks]{breakurl}
\usepackage{cleveref}
\usepackage{tikz}
\usetikzlibrary {quotes} 

\hypersetup{
  colorlinks   = true,
  urlcolor     = blue!50!black,
  linkcolor    = blue!50!black,
  citecolor   = blue!50!black
}

\newtheorem{theorem}{Theorem}
\newtheorem{lemma}[theorem]{Lemma}
\newtheorem{claim}{Claim}
\newtheorem{observation}[theorem]{Observation}
\newtheorem{conjecture}[theorem]{Conjecture}
\newtheorem{question}[theorem]{Question}
\newtheorem{problem}[theorem]{Problem}

\renewcommand{\leq}{\leqslant}
\renewcommand{\geq}{\geqslant}
\newcommand*{\eps}{\varepsilon}
\newcommand*{\N}{\mathbb{N}}

\newcommand*{\abs}[1]{\left \lvert {#1} \right \rvert}
\newcommand*{\ceil}[1]{\left \lceil {#1} \right \rceil}
\newcommand*{\floor}[1]{\left \lfloor {#1} \right \rfloor}

\undef{\set}
\DeclarePairedDelimiter{\set}{\lbrace}{\rbrace}

\newcommand\blfootnote[1]{
  \begingroup
  \renewcommand\thefootnote{}\footnote{#1}
  \addtocounter{footnote}{-1}
  \endgroup
}

\advance\textwidth2\evensidemargin 
\title{\texorpdfstring{\vspace{-4ex}}{}A note on interval colourings of graphs}

\date{17 May 2023}

\author{
Maria Axenovich\footnote{Karlsruhe Institute of Technology, Karlsruhe, Germany. Research supported in part by DFG grant FKZ AX 93/2-1.} \;
Ant\'onio Gir\~ao\footnote{Mathematical Institute, University of Oxford,
Oxford OX2 6GG, UK.}~\footnote{Research supported by EPSRC grant EP/V007327/1.} \;
Lawrence Hollom\footnote{Department of Pure Mathematics and Mathematical Statistics (DPMMS), University of Cambridge, Wilberforce Road, Cambridge CB3 0WA, UK.} \;
Julien Portier\protect\footnotemark[4] \;\\
Emil Powierski\protect\footnotemark[2] \;
Michael Savery\protect\footnotemark[2]~\footnote{Heilbronn Institute for Mathematical Research, Bristol, UK.} \;
Youri Tamitegama\protect\footnotemark[2]\;
Leo Versteegen\protect\footnotemark[4]
}

\begin{document}
\maketitle

\begin{abstract}
\noindent A graph is said to be \emph{interval colourable} if it admits a proper edge-colouring using palette $\mathbb{N}$ in which the set of colours incident to each vertex is an interval. The \emph{interval colouring thickness} of a graph $G$ is the minimum $k$ such that $G$ can be edge-decomposed into $k$ interval colourable graphs. We show that $\theta(n)$, the maximum interval colouring thickness of an $n$-vertex graph, satisfies $\theta(n) =\Omega(\log(n)/\log\log(n))$ and $\theta(n)\leq n^{5/6+o(1)}$, which improves on the trivial lower bound and an upper bound of the first author and Zheng. As a corollary, we answer a question of Asratian, Casselgren, and Petrosyan and disprove a conjecture of Borowiecka-Olszewska, Drgas-Burchardt, Javier-Nol, and Zuazua. We also confirm a conjecture of the first author that any interval colouring of an $n$-vertex planar graph uses at most $3n/2-2$ colours.

\blfootnote{Email: \textsf{maria.aksenovich@kit.edu,
\{girao,powierski,savery,tamitegama\}@maths.ox.ac.uk,\\\{lh569,jp899,lvv23\}@cam.ac.uk}}
\end{abstract}

\section{Introduction}\label{sec:intro}
We say that a graph $G$ is \emph{interval colourable} if it has an \emph{interval colouring} $c$, that is, a proper edge-colouring $c \colon E(G) \to \N$ such that the set of colours incident to each vertex $v \in V(G)$, $\{c(vw) \colon w \in N(v)\}$, consists of consecutive integers.  Since these notions were introduced by Asratian and Kamalian \cite{AK1} in 1987, interval colourable graphs and their properties have been studied extensively (see, for example, \cite{ACP, AK1, AK2, BDJZ, H, AZ, S}). Examples of interval colourable graphs include trees and regular class 1 graphs (and hence, in particular, regular bipartite graphs), while non-interval colourable graphs include odd cycles and complete graphs with an odd number of vertices (see~\cite{S} for a further, bipartite, example).

In the present paper, we principally study a parameter introduced recently by Asratian, Casselgren, and Petrosyan~\cite{ACP} to quantify how far a graph is from being interval colourable.
The \emph{interval colouring thickness} of a graph $G$, denoted $\theta(G)$, is the minimum number $k$ such that $G$ can be edge-decomposed into $k$ interval colourable subgraphs.
Interval colouring thickness, and interval colourings more generally, are of particular interest in theoretical computer science as they are related to scheduling tasks without waiting periods. For instance, suppose one wishes to schedule one on one meetings between certain attendees at a conference; let $G$ be the graph with the attendees as vertices, and edges corresponding to the desired meetings. If each meeting lasts the same amount of time and nobody is willing to wait between any of their meetings on a given day, then the interval thickness of $G$ corresponds to the minimum number of days over which the meetings must run.

 We write $\theta(n)$ for the maximum interval colouring thickness of a graph on $n$ vertices, and $\theta'(m)$ for the maximum interval colouring thickness of a graph with $m$ edges. Upper bounds on $\theta(G)$ in terms of the edge chromatic number of $G$ obtained by Asratian, Casselgren, and Petrosyan~\cite{ACP} imply that $\theta(n)\leq 2 \ceil{n/5}$, and the same authors observe that an arboricity result of Dean, Hutchinson, and Scheinerman \cite{DHS} gives $\theta'(m) \leq \ceil{\sqrt{m/2}}$.  The first author and Zheng~\cite{AZ} improved on the first of these bounds, showing that $\theta(n)$ is sublinear. In Section~\ref{sec:ub} we show that a result of R\"odl and Wysocka~\cite{RW} implies the following polynomial improvement.

\begin{theorem}\label{thm:ub}
We have  $\theta(n) \leq n^{5/6+o(1)}$ and  $\theta'(m) \leq m^{5/11+o(1)}$.
\end{theorem} 

As noted above, various graphs with $\theta(G) \geq 2$ are known, but to the best of our knowledge none with $\theta(G) \geq 3$ have been found.
In Section~\ref{sec:lb} we show that $\theta(n)=\omega(1)$.

\begin{theorem}\label{thm:lb}
We have $\theta(n)= \Omega\left(\frac{\log(n)}{\log\log(n)}\right)$, and consequently $\theta'(m)= \Omega\left(\frac{\log(m)}{\log\log(m)}\right)$.
\end{theorem}

The first bound in Theorem~\ref{thm:lb} is proved via a random construction, which, in fact, produces a bipartite graph. The bound on $\theta'(m)$ follows trivially (see Section~\ref{sec:lb}). One corollary of this theorem is that for every integer $k$ there exists a graph with interval colouring thickness~$k$, answering a question of Asratian, Casselgren, and Petrosyan~\cite{ACP}. Indeed, given a graph $G$ of some interval colouring thickness $K$, and an edge-decomposition of $G$ into $K$ interval colourable subgraphs, the union of any $k\leq K$ of these subgraphs has interval colouring thickness $k$.

Theorem~\ref{thm:lb} also disproves a conjecture of Borowiecka-Olszewska, Drgas-Burchardt, Javier-Nol, and Zuazua \cite{BDJZ}, who defined an oriented graph to be \emph{consecutively colourable} if it has an arc colouring using palette $\N$ such that for each vertex $v$, the colours of the out-arcs from $v$ are all different and form an interval, and similarly for the in-arcs to $v$.
They conjectured that for every graph $G$, there exists an orientation of its edges that is consecutively colourable. It is easy to see that this conjecture implies that $\theta(G) \leq 2$ for every bipartite graph $G$, and hence is false by our construction.

In the last part of this paper, rather than studying the minimum number of interval colourable graphs required to edge-decompose a graph, we are instead interested in the maximum number of colours which can be used in an interval colouring of a given interval colourable graph. 
To this end, for each interval colourable graph $G$ we define $t(G)$ to be the greatest number of colours used in an interval colouring of $G$. This parameter was introduced by Asratian and Kamalian~\cite{AK1} who proved that $t(G) \leq \abs{V(G)}-1$ if $\abs{V(G)}\geq 1$ and $G$ contains no triangles and later~\cite{AK2} that $t(G) \leq 2\abs{V(G)}-1$ for all graphs $G$ with $\abs{V(G)}\geq 1$. This was improved by Giaro, Kubale, and Ma\l afiejski~\cite{GKM} to $t(G) \leq 2\abs{V(G)}-4$ for all graphs with at least three vertices. The first author~\cite{Axenovich2002} improved this bound over the class of planar graphs to $t(G) \leq (11/6)\abs{V(G)}$ and conjectured that this could be further improved to $t(G) \leq (3/2)\abs{V(G)}$. We confirm this conjecture.

\begin{theorem}
\label{thm:3/2result}
Let $G$ be a planar graph on $n\geq 2$ vertices that admits an interval colouring. Then $t(G)\leq (3/2)n-2$.
\end{theorem}

Theorem~\ref{thm:3/2result} was shown to be tight in~\cite{Axenovich2002}. In Section~\ref{sec:planar}, where the theorem is proved, we recall the constructions demonstrating this fact and extend them to a larger collection of graphs.

\textbf{Note.}
This paper supersedes two independent works by subsets of the present authors which appeared on arXiv almost simultaneously~\cite{AGPST,HPV}. Shortly after this, an independent proof of Theorem~\ref{thm:3/2result} was announced by Arsen Hambardzumyan and Levon Muradyan~\cite{Ham}. 

\section{Interval colouring thickness lower bound}\label{sec:lb}

In order to provide a lower bound on the interval colouring thickness of the graph $G$ that we construct below, we will need to show that every edge-decomposition of $G$ into sufficiently few parts has a part which is not interval colourable. Our approach to disproving interval colourability is based on the following observation, noted by Sevastianov \cite{S}.  
\begin{observation}\label{obs:low_deg_path}
Let $G$ be an interval colourable graph and let $U \subseteq V(G)$. Suppose that there exists $d\in \N$ such that for all distinct $v,w \in U$ there is a path $P$ in $G$ from $v$ to $w$ such that $\sum_{x\in V(P)} d(x)\leq d.$ Then for all $u \in V(G)$, we have $\abs{N(u) \cap U}\leq d$.
\end{observation}
\begin{proof}
    If $\abs{U}\leq 1$ then the lemma holds trivially, so assume $\abs{U}\geq 2$. Let $c$ be an interval colouring of $G$, let $u\in V(G)$, and let $v,w\in N(u)\cap U$ be distinct. Fix a path $P=x_1\dots x_k$ in $G$ from $v=x_1$ to $w=x_k$ such that $\sum_{i=1}^k d(x_i)\leq d$. The colours under $c$ of any two edges incident to a vertex $x$ of $G$ differ by at most $d(x)-1$, so it follows from the existence of $P$ that any edge incident to $v$ and any edge incident to $w$ have colours differing by at most $\sum_{i=1}^k (d(x_i)-1)\leq d-1$. Observe that we can choose $v$ and $w$ such that $c(uv)-c(uw)\geq \abs{N(u) \cap U}-1$ which gives $\abs{N(u) \cap U} \leq d$.
\end{proof}

We shall construct a bipartite graph of large interval colouring thickness by patching together many bipartite graphs of the form given by the following lemma.

\begin{lemma}\label{lem:low_diam}
Fix $\alpha\in(0,1/2]$ and let $a$ and $n$ be integers satisfying $n\geq \max\{1000(\log(a)+1)/\alpha, a+1\}$. Then there is a bipartite graph $G$ on parts $A$ and $B$ of sizes $a$ and $n$ respectively satisfying the following.
\begin{enumerate}[label=(\alph*), font=\textup]
    \item For all $x \in A$, $ d(x)=\floor{\alpha n}$.  \label{it:degrees}
    \item For each $\delta \in (0,1]$ with $\delta \geq 10 a^{-1/3}\alpha^{-1}$, if $H$ is a subgraph of $G$ with at least $\alpha\delta a n$ edges, then there exist $A' \subseteq A$ and $B' \subseteq B$ with $\abs{A'} \leq 1/\alpha$ and $\abs{B'} \geq  \delta n/16$ such that the induced subgraph $H[A' \cup B']$ has diameter at most $6$. \label{it:diam}
\end{enumerate}
\end{lemma}
\begin{proof}
Let $A$ and $B$ be vertex sets of the desired sizes and consider the random bipartite graph between them in which every edge appears independently with probability $4\alpha/3$. Drawing on $n\geq 1000(\log(a)+1)/\alpha$, a standard application of a Chernoff bound and a union bound over all vertices $x \in A$ yields that with failure probability at most $1/e$ we have $d(x) \geq \alpha n$ for all $x \in A$. 

For a fixed set $U\subseteq B$ and $y\in A$, let $E_{U,y}$ be the event that $|U\cap N(y)|>2\alpha |U|$. We will say that property~$(\star)$ holds in this random graph if for every $U\subseteq B$ there are at most $m \coloneqq 24n/\alpha |U|$ vertices $y\in A$ for which $E_{U,y}$ occurs.

\begin{claim}
Property~$(\star)$ holds with failure probability at most $1/e$.
\end{claim}
\begin{proof}
Fix $U\subseteq B$. For each $y\in A$, a standard application of a Chernoff bound yields that $E_{U,y}$ occurs with probability at most $e^{-\frac{1}{12}\alpha|U|}$. Since these events are independent, for any set $W\subseteq A$ the event $\bigcap_{y\in W} E_{U,y}$ occurs with probability at most $e^{-\frac{1}{12}\alpha |U||W|}$.
Taking a union bound over all such subsets $W$ with $|W|\geq m$, we find that with failure probability at most $2^ae^{-2n}$ there are at most $m$ vertices $y\in A$ for which $E_{U,y}$ occurs.
The claim follows from a second union bound over all $U\subseteq B$, using that $n\geq a+1$.
\end{proof}

Thus, there is a positive probability that in this random graph we have both $d(x) \geq \alpha n$ for all $x \in A$, and property ($\star$). Given an outcome with these properties, we can delete edges where necessary to obtain a graph $G$ satisfying~\ref{it:degrees} and~($\star$). We will now show that $G$ also satisfies~\ref{it:diam}.

Let $\delta \in (0,1]$ with $\delta \geq 10 a^{-\frac{1}{3}}\alpha^{-1}$ and fix a subgraph $H$ of $G$ with at least $\alpha\delta a n$ edges. Let $\{s_1,\ldots, s_k\}\subseteq A$ be a maximal set with the property that there is a collection of pairwise  disjoint sets $\{S_1,\ldots, S_k\}$ such that $S_i\subseteq N_{H}(s_i)$ and $\abs{S_i}=\ceil{\alpha\delta n/4}$ for each $i \in [k]$. Note that $k \leq 4/(\alpha\delta)$ (indeed, otherwise $n=\abs{B}\geq \sum_i \abs{S_i} >n$). By property ($\star$), for each $i\in[k]$ there are at most $m=24n/(\alpha |S_i|)$ vertices $y\in A$ with $\abs{S_i\cap N_G(y)}> 2\alpha\abs{S_i}$. Since the degree of each vertex in $A$ is at most $\floor{\alpha n}$, these vertices are incident to at most $km\floor{\alpha n}$ edges in $H$. It follows, using the fact that $a\geq 1000/(\alpha^3\delta^3)$, that there exists a vertex $z\in A$ in $H$ which has $\abs{S_i\cap N_G(z)}\leq 2\alpha\abs{S_i}$ for all $i\in [k]$ with
\[
d_H(z)\geq \frac{e(H)-km\floor{\alpha n}}{a} \geq \frac{\alpha \delta a n-4\alpha^{-1}\delta^{-1}\cdot 96\alpha^{-2}\delta^{-1}\alpha n}{a}=\alpha\delta n - \frac{384n}{\alpha^2\delta^2a} \geq \frac{\alpha\delta n}{2}.
\]

Let $S= \bigcup S_i$. If $z\in \{s_1,\dotsc,s_k\}$, then clearly $|N_{H}(z)\cap S|\geq \alpha\delta n/4$.
Otherwise, if $z\in A\setminus \{s_1,\dotsc,s_k\}$, yet $\abs{N_{H}(z)\cap (B\setminus S)}\geq  \alpha\delta n/4 $, then we could add $z$ to $\{s_1,\dotsc, s_k\}$, contradicting the maximality of that set. Hence, in either case $\abs{N_{H}(z)\cap S}\geq \alpha\delta n/4$.
The number of the $S_i'$ that have a non-empty intersection with $N_H(z)$ is at least $|N_{H}(z)\cap S|/ \max \{|N_H(z)\cap S_i|: i=1, \ldots, k\}$, which is thus
at least $\frac{\alpha\delta n/4}{2\alpha\ceil{\alpha\delta n/4}} \geq \frac{1}{4\alpha}$.

Denote by $B'$ the union of the first $\ceil{1/(4\alpha)}$ of the $S_i$ which intersect $N_{H}(z)$ and let $A'$ be the set containing $z$ and the corresponding $s_i$.
By construction, $H[A'\cup B']$ has diameter at most $6$ and we have $\abs{A'}= \ceil{1/(4\alpha)}+1 \leq 1/\alpha$ 
and $\abs{B'}\geq \ceil{1/(4\alpha)} \cdot \alpha\delta  n/4 \geq \delta n/16$, as required.
\end{proof}

To obtain Theorem~\ref{thm:lb}, we will construct a graph $G$ that is a union of $t \approx \log(n)$ bipartite graphs $G_i$ with bipartitions $(A_i, B)$ such that $|A_i|= \sqrt{n}$ and $|B|=n$, and where the degrees of the vertices in $A_{i-1}$ are twice those of the vertices in $A_i$. The graphs $G_i$ are obtained by $t$ separate applications of Lemma~\ref{lem:low_diam} so that each of them satisfies~\ref{it:diam} and we suppose towards a contradiction that $G$ can be edge-partitioned into $\ell \approx \log(n)/\log\log(n)$ interval colourable subgraphs. One of these subgraphs $H$ has $\sum_{i} e(H \cap G_i)/e(G_i)\geq t'/\ell\approx \log\log(n)$. Letting $\delta_i\coloneqq e(H\cap G_i)/e(G_i)$, we restrict our attention to those $i\in [t]$ for which $\delta_i$ is large enough to invoke property~\ref{it:diam} for $H\cap G_i$, noting that this does not result in a significant loss of total edge density.

We now apply property~\ref{it:diam} for each remaining $i$ to find a small diameter subgraph $G'_i$ of $H \cap G_i$ on $A_i' \cup B_i$ where $|A_i'|$ is small and $|B_i|\gtrsim \delta_in$. Restricting the indices under consideration once more to some set $I\subseteq [t]$, we can achieve that for $i<j\in I$, the degrees in $A'_i$ are much bigger than those in $A'_{j}$, while maintaining that $\sum_{i} \delta_i\geq C$ for some constant $C$. We can now use Observation~\ref{obs:low_deg_path} to show that $\abs{B_i\cap B_j}$ is small for all $i\neq j$. Combined with the fact that the sum of the sizes of the $B_i$ is much larger than $n$, this gives the required contradiction.

\begin{proof}[Proof of Theorem~\ref{thm:lb}]
We may assume that $n$ is large, so in what follows we will not concern ourselves with whether expressions are integers. Let $\ell=\log(n)/(280\log\log(n))$. To prove the theorem we will construct a graph $G$ on at most $2n$ vertices which we will show has interval colouring thickness greater than $\ell$. Let $t=\log(n)/7$ and let $B$ be a set of $n$ vertices. For each $i \in [t]$, let $A_i$ be a set of $a =\sqrt{n}$ vertices and let $G_i$ be a bipartite graph with bipartition $(A_i,B)$ of the form given by Lemma~\ref{lem:low_diam}, where $\alpha$ is taken to be $\alpha_i=2^{-i}$. To see that the assumption on $n$ in the statement of 
the lemma is satisfied in this setting, note that $n^{-1/7}\leq \alpha_i \leq 1/2$ and that since $n$ is large, $n\geq a+1$ and
\[
\frac{1000(\log(a)+1)}{\alpha_i}\leq 1000\cdot n^{1/7}\cdot\left(\frac{\log(n)}{2}+1\right)\leq n.
\]

Let $G$ be the union of the $G_i$, so that $G$ is a bipartite graph with bipartition $(\bigcup_i A_i,B)$. Note that $|V(G)|= n+ t\sqrt{n} = n + \log(n) \sqrt{n}/7 \leq 2n$, and suppose towards a contradiction that $G$ has an edge-decomposition into $\ell$ interval colourable subgraphs.
For each subgraph $H$ in such a decomposition, and each $i\in [t]$, let $\delta_i(H)\coloneqq\abs{E(H)\cap E(G_i)}/e(G_i)$. Note that
\[
\sum_{H} \sum_{i\in [t]} \delta_i(H)= t,
\]
so by the pigeonhole principle there is a subgraph $H$ in the decomposition such that $\sum_{i\in [t]} \delta_i(H)\geq t/\ell\geq 40\log\log(n)$.

By another application of the pigeonhole principle we can find $I'\subseteq [t]$ such that $i + 2\log\log(n) \leq j$ for all $i<j$ in $I'$ and $\sum_{i\in I'} \delta_i(H) \geq 20$. Now form $I\subseteq I'$ by deleting any $i$ for which $\delta_i(H)<1/\log(n)$, and note that since $t\leq \log(n)$, we still have $\sum_{i\in I} \delta_i(H) \geq 19$. For each $i\in I$ we have $\delta_i(H)\geq 1/\log(n)\geq 10\cdot n^{-1/6}\cdot n^{1/7} \geq 10a^{-1/3}\alpha_i^{-1}$, so by the construction of $G_i$ there exist sets $A'_i\subseteq A_i$ and $B_i\subseteq B$ such that $\abs{A'_i}\leq \alpha_i^{-1}$, $\abs{B_i}\geq \delta_i(H)n/16$, and $H[A'_i\cup B_i]$ has diameter at most 6. Next, we show that the intersection of any two of the sets $B_i$ is small.

\begin{claim}
    For each pair $j<i$ in $I$, we have $\abs{B_i \cap B_j}\leq 6\alpha_i\alpha_j^{-1}n$.
\end{claim}

\begin{proof}
    We will apply Observation~\ref{obs:low_deg_path} to the interval colourable graph $H$, where $U$ is taken to be $A'_i\cup B_i$. For each $x\in A'_i$ we have $d_H(x)\leq \alpha_i n$, for each $y\in B_i$ we have $d_H(y)\leq t\sqrt{n} \leq \alpha_i n$, and $H[A'_i\cup B_i]$ has diameter at most 6, so between any two points in $A'_i\cup B_i$ there is a path in $H$ whose degree sum is at most $6\alpha_i n$. It follows that for all $v\in A'_j$, we have $\abs{N_H(v)\cap B_i}\leq 6\alpha_i n$. Finally, since $B_j\subseteq N_H(A'_j)$, we have $\abs{B_i \cap B_j}\leq \sum_{v\in A_j'} |N_H(v)\cap B_i|\leq  \abs{A'_j} 6\alpha_i n \leq 6\alpha_i\alpha_j^{-1}n$.
\end{proof}

By the construction of $I$, for $j<i$ in $I$ we have $\alpha_i\alpha_j^{-1}\leq 2^{-2\log\log(n)}$, so by the Claim 2 and recalling that $|I| \leq t = \log (n)/7$, we have
\[
\sum_{\substack{
j<i\\
i,j\in I
}}{\abs{B_i\cap B_j}}\leq \abs{I}^2 \cdot 6\cdot 2^{-2\log\log(n)}n \leq 6t^2\log(n)^{-2} n< \frac{n}{8}.
\]

It follows that
\[
n\geq \abs{\bigcup_{i\in I}{B_i}}\geq \sum_{i\in I}{\abs{B_i}}-\sum_{\substack{
j<i\\
i,j\in I
}}{\abs{B_i\cap B_j}}
\geq \sum_{i\in I}{\frac{\delta_i(H)n}{16}}-\frac{n}{8}
\geq \frac{19n}{16}-\frac{n}{8}
> n
\]
which gives the required contradiction and completes the proof of the bound on $\theta(n)$. It is straightforward to deduce the bound on $\theta'(m)$ by noting that any $n$-vertex graph contains at most $n^2$ edges and $\theta'$ is a non-decreasing function, so $\theta'(n^2)\geq \theta(n)$ for all~$n$. 
\end{proof} 

\section{Interval colouring thickness upper bound}\label{sec:ub}

We prove Theorem~\ref{thm:ub} using an edge-decomposition into forests and regular bipartite subgraphs. Such graphs are interval colourable (regular bipartite graphs can be edge-decomposed into perfect matchings which can be taken as colour classes). To find large regular subgraphs of relatively dense graphs we rely on the following result of R\"odl and Wysocka~\cite{RW}.
\begin{theorem}[\cite{RW}]\label{thm:RW}
Let $\gamma\colon \N\to [0,1/2)$ satisfy $\gamma(n)=\omega(n^{-1/3})$ as $n\to\infty$. Then every $n$-vertex graph with at least $\gamma n^2$ edges contains an $\Omega(\gamma^3 n)$-regular subgraph.
\end{theorem}

Before proving Theorem~\ref{thm:ub}, we note that it is relatively straightforward to use Theorem~\ref{thm:RW} and the bound $\theta'(m)=O\left(\sqrt{m}\right)$~\cite{ACP,DHS} mentioned in the introduction to prove that $\theta(n)=n^{1-\Omega(1)}$. Indeed, let $G$ be an $n$-vertex graph and suppose that $e(G)=\Omega(n^{2-2/13})$. Then $G$ contains a bipartite subgraph with $\Omega(n^{2-2/13})$ edges, which in turn
contains an $\Omega(n^{1-6/13})$-regular (bipartite) subgraph by Theorem~\ref{thm:RW}. Such a subgraph contains $\Omega(n^{2-12/13})$ edges. 
 Delete these edges from~$G$, and repeat until the remaining graph has fewer than $n^{2-2/13}$ edges. This process terminates after $O(n^{2-2+12/13})=O(n^{12/13})$ steps, at which point we have partitioned $G$ into $O(n^{12/13})$ regular bipartite (and hence interval colourable) subgraphs, and a subgraph with $O(n^{2-2/13})$ edges. This final portion has interval colouring thickness $O(\sqrt{n^{2-2/13}})=O(n^{12/13})$ by the bound on $\theta'(m)$, so $\theta(G)=O(n^{12/13})$.

When upper bounding the number of steps in this process we assumed the worst: that there were on the order of $n^2$ edges in $G$ to begin with. If this were the case, however, then we could (and should) have taken the first few regular bipartite subgraphs to have more edges than we did. Thus, to make optimal use of Theorem~\ref{thm:RW} we must adjust the sizes of the regular bipartite subgraphs as we go through. Note that at the end of the process we appeal to a bound on $\theta'(m)$, so the quality of the bound on $\theta(n)$ that we can obtain from this argument will depend on the quality of our $\theta'(m)$ bound. The next lemma optimises this approach for any given polynomial upper bound on $\theta'(m)$.

\begin{lemma}\label{lem:vtx}
Suppose that $\theta'(m)= O(m^{\beta})$ for some $\beta\in(0,1/2]$. Then $\theta(n)\leq n^{\alpha+o(1)}$, where~$\alpha=\frac{10\beta}{5+ \beta}$.
\end{lemma}

\begin{proof}
Fix $\eps>0$. It is sufficient to show that every $n$-vertex graph has interval colouring thickness $O(n^{\alpha+\eps})$. Let $\alpha_i=\frac{\alpha}{5}\left(1-\frac{1}{6^i}\right)$ for each $i\in\N$, let $\alpha_0=0$, and suppose that $G$ is an $n$-vertex graph with $n^{2-\alpha_{i+1}}\leq e(G)\leq n^{2-\alpha_i}$ for some $i\in\N_0$. Then $G$ contains a bipartite subgraph with at least $n^{2-\alpha_{i+1}}/2$ edges which in turn has an $\Omega(n^{1-3\alpha_{i+1}})$-regular (bipartite) subgraph by Theorem~\ref{thm:RW} (note that $\alpha_{i+1}<1/3$). This subgraph has $\Omega(n^{2-6\alpha_{i+1}})$ edges, so if we delete these edges from $G$ and repeat until the remaining graph has fewer than $n^{2-\alpha_{i+1}}$ edges, then the process will terminate after $O(n^{2-\alpha_i}/n^{2-6\alpha_{i+1}})=O(n^{6\alpha_{i+1}-\alpha_i})=O(n^{\alpha})$ steps.

Let $C\in\N$ be large enough that $\alpha_C\geq \frac{\alpha}{5}-\beta^{-1}\eps$. Every $n$-vertex graph $G$ has at most $n^2=n^{2-\alpha_0}$ edges, so by repeatedly applying the above, we can remove at most $C\cdot O(n^{\alpha})=O(n^{\alpha})$ regular bipartite graphs from $G$ so that fewer than $n^{2-\alpha_C}$ edges remain. By the assumed bound on $\theta'(m)$, this remaining portion has interval colouring thickness $O(n^{\beta(2-\alpha_C)})=O(n^{\beta(2-\alpha/5+\beta^{-1}\eps)})=O(n^{\alpha+\eps})$.
\end{proof}

The next lemma converts an upper bound on $\theta(n)$ to one on $\theta'(m)$. The proof is based on the idea that every connected graph with $m$ edges is either dense enough that we can apply the bound on $\theta(n)$, or sparse enough that a spanning tree contains a large proportion of its edges.

\begin{lemma}\label{lem:edg}
Suppose that $\theta(n)=O(n^{\alpha})$ for some $\alpha\in(0,1]$. Then $\theta'(m)=O(m^{\beta})$
where $\beta=\frac{\alpha}{1+\alpha}$.
\end{lemma}
\begin{proof}
Let $\lambda\geq 1$ be such that $\theta\left(\ceil{\frac{1}{\beta}x}\right)\leq \lambda x^\alpha$ for all $x \in [1,\infty)$. We will show that $\theta(G)\leq \lambda m^{\beta}$ for all graphs $G$ with $e(G)=m$. We proceed by induction on $m$. First, observe that if $G$ has connected components $C_1, \dots, C_k$, then $\theta(G)=\max_i \theta(C_i)$, so we may assume that $G$ is connected. If $\abs{V(G)}> m' \coloneqq \ceil{\frac{1}{\beta}m^{1-\beta}}$, then $G$ contains a tree $T$ with $m'$ edges. By induction and applying the mean value theorem to the function $x \mapsto x^\beta$, we obtain
\[
\theta(G-T)\leq \lambda(m-m')^{\beta}\leq\lambda (m^{\beta} - m'\beta m^{\beta-1})\leq \lambda(m^{\beta}-1)\leq \lambda m^\beta-1
\]
since $\lambda\geq 1$. As $T$ is interval colourable we have $\theta(G) \leq \lambda m^{\beta}$ as required. 
On the other hand, if $\abs{V(G)} \leq m'$, then since $\theta$ is monotonic we have $\theta(G)\leq \theta\left(m'\right)\leq \lambda m^{\alpha-\beta \alpha}=\lambda m^{\beta}$ by our choice of~$\lambda$.
\end{proof}

We can now prove Theorem~\ref{thm:ub} by playing the preceding lemmas off against each other. 

\begin{proof}[Proof of Theorem~\ref{thm:ub}] 
Let $\alpha_0=1$, then for each $i \in \N$ let $\beta_{i}=\frac{\alpha_{i-1}}{1+\alpha_{i-1}}$ and $\alpha_{i}=\frac{10 \beta_i}{5+\beta_i}=\frac{10 \alpha_{i-1}}{5+6\alpha_{i-1}}$. Noting that $\beta_i=\frac{10\beta_{i-1}}{5+11\beta_{i-1}}$ for $i\geq 2$, we see that $\alpha_i\in(0,1)$ and $\beta_i\in(0,1/2]$ for all $i\in\mathbb{N}$. Hence, if $\theta(n) \leq n^{\alpha_{i}+o(1)}$ for some $i\in\mathbb{N}_0$, then Lemma~\ref{lem:edg} yields $\theta'(m)\leq m^{\beta_{i+1}+o(1)}$. This allows us to deduce from Lemma~\ref{lem:vtx} that $\theta(G) \leq n^{\alpha_{i+1}+o(1)}$. Clearly $\theta(n) \leq n=n^{\alpha_0}$, so $\theta(n)\leq n^{\alpha_i+o(1)}$ and $\theta'(m)\leq m^{\beta_i+o(1)}$ for all $i\in\mathbb{N}$. The theorem now follows from the fact that $\alpha_i \to \frac{5}{6}$ and $\beta_i \to \frac{5}{11}$ as $i \to \infty$. 
\end{proof}

\section{The maximal number of colours in an interval colouring of a planar graph}
\label{sec:planar}

In this section we will prove Theorem~\ref{thm:3/2result} via the following slightly more general result.

\begin{theorem}
\label{thm:generalisation-interval-colourable}
    Let $k\in\mathbb{R}_{\geq 0}$ and let $G$ be a graph on $n\geq 2$ vertices such that every subgraph $H\subseteq G$ on at least three vertices satisfies $|E(H)|\leq k(|V(H)|-2)$. 
    If $G$ admits an interval colouring, then $t(G)\leq (k/2)n + 1 - k$.
\end{theorem}

Note that a subgraph of a planar graph is also a planar graph, and that any planar graph $H$ on at least three vertices satisfies $|E(H)|\leq 3(|V(H)|-2)$, so plugging $k=3$ into \Cref{thm:generalisation-interval-colourable} gives \Cref{thm:3/2result}. 
We observe that for $k<1$ the result is trivial, since the condition on $G$ implies (for $n\geq 3$) that $G$ has no edges. Also, for $1\leq k<2$, the condition implies that $G$ is a matching, so the result is again trivial. For $2\leq k<3$ the graph $G$ is triangle-free, so the result is superseded by the result of Asratian and Kamalian~\cite{AK1} stating that such graphs have $t(G) \leq \abs{V(G)}-1$. Finally, for $k\geq 4$ \Cref{thm:generalisation-interval-colourable} is beaten by the result of Giaro, Kubale, and Ma\l afiejski~\cite{GKM} that $t(G) \leq 2\abs{V(G)}-4$.

\begin{proof}[Proof of \Cref{thm:generalisation-interval-colourable}.]
Fix $k\geq 0$ and
assume for a contradiction that $G$ is a counterexample to the theorem on the fewest possible vertices.
Let $n=\abs{V(G)}$ and let $G$ be interval coloured with colours $1,\dotsc,t$ (all used at least once), where $t=t(G)\geq 2$. Note that clearly $n\geq 3$.

We first claim that there is a colour $c$ with $1 < c < t$ such that there is a unique edge of $G$ of colour $c$. Indeed, if this were not the case, then all colours except perhaps 1 and $t$ would occur at least twice, and thus $e(G)\geq 2(t-2)+2=2t-2> k(n-2)$. This contradicts the assumption that $e(G) \leq k(n-2)$.

Hence, let $vw$ be the unique edge of colour $c$.
 Let $V_1$ be the set of vertices in $V(G)\setminus\{v,w\}$ that are incident only to edges of colours smaller than $c$ and 
  let $V_2$ be the set of vertices in $V(G)\setminus\{v,w\}$ that are incident only to edges of colours greater than $c$. 
  We see that $V(G) = V_1\cup V_2 \cup \{v, w\}$. Indeed, otherwise there is a vertex $y\in V(G)\setminus\{v,w\}$ that is incident to an edge of colour greater than $c$ and to an edge of colour less than $c$. Since $vw$ is the only edge of colour $c$, 
  $y$ is not incident to an edge of colour $c$ and the set of colours on edges incident to $y$ do not form an interval. 
From the definition of $V_1$ and $V_2$ we see that there are no edges between these two sets. Hence, the induced colouring on $G_i=G[V_i\cup\{v,w\}]$ is an interval colouring for $i=1,2$.

Note that $\abs{V(G_1)}+\abs{V(G_2)}=n+2$ and $t(G_1)+t(G_2) \geq t+1$. 
Since $1<c<t$, we have $V_1, V_2\neq \emptyset$, and thus $3\leq |V(G_i)|<n$ for $i=1,2$. By the minimality of $G$, it follows that
$t(G_i)\leq (k/2)\abs{V(G_i)} +1 -k$ for $i=1,2$. Thus
 $t \leq t(G_1)+t(G_2)-1\leq (k/2)(\abs{V(G_1)}+\abs{V(G_2)})+1-2k=(k/2)n+1-k$,  a contradiction.
\end{proof}

The bound in \Cref{thm:3/2result} is attained for even $n$
by any graph $G_{2s}$ of the form shown in Figure~\ref{fig:maximal-example-labelled}. This figure also demonstrates how to interval colour these graphs. To see that they are indeed planar, consider the drawing as shown in Figure~\ref{fig:maximal-example-planar}, where the dashed lines marked in red should be drawn to loop around the left-hand side of the graph.
In fact, more extremal graphs can be obtained by removing any set of dashed edges shown in blue in Figure~\ref{fig:maximal-example-labelled} (the extremal examples from \cite{Axenovich2002} correspond to including no blue edges).
Indeed, any such graph is interval colourable since in the colouring shown in the figure, the colours of the blue edges are either the minimal or maximal colour at every vertex they are incident to. 

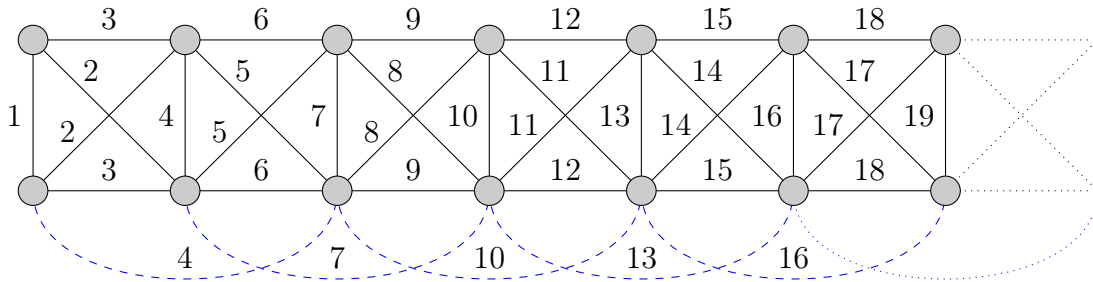
\begin{figure}[ht]
\centering
\begin{tikzpicture}[scale=2,darkstyle/.style={circle,draw,fill=gray!40,minimum size=10}]
\path[-] 
    (1,1) edge ["1"] (1,2) 
        edge ["3"] (2,1) 
        edge [pos=0.3, inner sep=1] ["2"] (2,2)
    (2,1) edge ["4"] (2,2) 
        edge ["6"] (3,1) 
        edge [pos=0.3, inner sep=1] ["5"] (3,2)
    (3,1) edge ["7"] (3,2) 
        edge ["9"] (4,1) 
        edge [pos=0.3, inner sep=1] ["8"] (4,2)
    (4,1) edge ["10"] (4,2) 
        edge ["12"] (5,1) 
        edge [pos=0.35, inner sep=1] ["11"] (5,2)
    (5,1) edge ["13"] (5,2) 
        edge ["15"] (6,1) 
        edge [pos=0.35, inner sep=1] ["14"] (6,2)
    (6,1) edge ["16"] (6,2) 
        edge ["18"] (7,1) 
        edge [pos=0.35, inner sep=1] ["17"] (7,2)
    (7,1) edge ["19"] (7,2)
        edge [dotted] (8,1)
        edge [dotted] (8,2)
    (1,2) edge ["3"] (2,2) 
        edge [pos=0.3, inner sep=1.5] ["2"] (2,1)
    (2,2) edge ["6"] (3,2) 
        edge [pos=0.3, inner sep=1.5] ["5"] (3,1)
    (3,2) edge ["9"] (4,2) 
        edge [pos=0.3, inner sep=1.5] ["8"] (4,1)
    (4,2) edge ["12"] (5,2) 
        edge [pos=0.3, inner sep=1.5] ["11"] (5,1)
    (5,2) edge ["15"] (6,2) 
        edge [pos=0.3, inner sep=1.5] ["14"] (6,1)
    (6,2) edge ["18"] (7,2) 
        edge [pos=0.3, inner sep=1.5] ["17"] (7,1)
    (7,2) edge [dotted] (8,2)
        edge [dotted] (8,1)
    (1,1) edge [dashed, bend right=90, draw=blue, "4"] (3,1)
    (2,1) edge [dashed, bend right=90, draw=blue, "7"] (4,1)
    (3,1) edge [dashed, bend right=90, draw=blue, "10"] (5,1)
    (4,1) edge [dashed, bend right=90, draw=blue, "13"] (6,1)
    (5,1) edge [dashed, bend right=90, draw=blue, "16"] (7,1)
    (6,1) edge [dashed, bend right=90, draw=blue, dotted] (8,1);
\foreach \x in {1,...,7} 
    \foreach \y in {1,2} 
        {\node [darkstyle] (\x\y) at (\x,\y) {};}
\end{tikzpicture}
\caption{A planar graph $G_{2s}$ on $2s$ vertices attaining the maximum value of $t(G)$, shown here for $s=7$. Any subset of the blue dashed edges can be removed to find another graph attaining this maximum.}
\label{fig:maximal-example-labelled}
\end{figure}

\begin{figure}[ht]
\centering
\begin{tikzpicture}[scale=2,darkstyle/.style={circle,draw,fill=gray!40,minimum size=10}]
\path[-] 
    (1,1) edge (1,2) 
        edge (2,1) 
        edge (2,2)
    (2,1) edge [draw=red,dashed] (2,2) 
        edge (3,1) 
        edge [draw=red,dashed] (3,2)
    (3,1) edge  (3,2) 
        edge (4,1) 
        edge (4,2)
    (4,1) edge [draw=red,dashed] (4,2) 
        edge (5,1) 
        edge [draw=red,dashed] (5,2)
    (5,1) edge (5,2) 
        edge (6,1) 
        edge (6,2)
    (6,1) edge [draw=red,dashed] (6,2) 
        edge (7,1) 
        edge [draw=red,dashed] (7,2)
    (7,1) edge (7,2)
    (1,2) edge (2,2) 
        edge [draw=red,dashed] (2,1)
    (2,2) edge (3,2) 
        edge (3,1)
    (3,2) edge (4,2) 
        edge [draw=red,dashed] (4,1)
    (4,2) edge (5,2) 
        edge (5,1)
    (5,2) edge (6,2) 
        edge [draw=red,dashed] (6,1)
    (6,2) edge (7,2) 
        edge (7,1)
    (1,1) edge [bend left=23] (3,1)
    (2,1) edge [bend right=45] (4,1)
    (3,1) edge [bend left=23] (5,1)
    (4,1) edge [bend right=45] (6,1)
    (5,1) edge [bend left=23] (7,1);
\foreach \x in {1,...,7} 
    \foreach \y in {1,2} 
        {\node [darkstyle] (\x\y) at (\x,\y) {};}
\end{tikzpicture}
\caption{An illustration to show that $G_{2s}$ is planar: the red dashed edges should loop around the left-hand side of the graph.}
\label{fig:maximal-example-planar}
\end{figure}
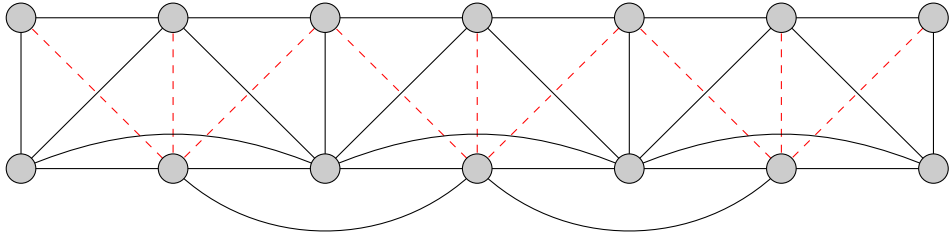

To see that the bound in \Cref{thm:3/2result} is tight (up to rounding) for odd $n$, take any extremal graph $G$ on an even number of vertices, interval colour it using all the colours $1,\dots,t$ (where $t=t(G)$), find an edge $uv$ which receives colour $t$, and add a new vertex $w$ adjacent only to $v$. The interval colouring of the original graph can now be extended to an interval colouring of the new graph by assigning the new edge colour $t+1$.

\section{Concluding remarks}

We have improved the bounds on the maximum interval colouring thickness of a graph on $n$ vertices, showing that $c\log(n)/\log\log(n) \leq \theta(n) \leq n^{5/6+o(1)}$ for some constant $c>0$. Clearly there remains a large gap between the bounds, which it would be interesting to narrow. We are particularly interested to know whether $\theta(n)$ is polynomial in $n$, and we believe this is not the case.

\begin{conjecture}\label{conj:o1}
  It holds that $\theta(n)= n^{o(1)}$.
\end{conjecture}

We proved our upper bound on $\theta(n)$ by showing that any graph could be edge-decomposed into the desired number of regular bipartite graphs and forests (which are all interval colourable).
However, no bound of the form $\theta(n)=o(\sqrt{n})$ can be obtained by decomposing only into these two types of graphs, as the following example shows.

\begin{observation}
Every edge-decomposition of $K_{n,\sqrt{n}}$ into forests and regular graphs has at least $(1-o(1))\sqrt{n}$ parts.
\end{observation}

This suggests that to further improve the upper bound on $\theta(n)$, it might be useful to prove an upper bound on the interval colouring thickness of biregular graphs (where a graph is \emph{biregular} if it is bipartite and all vertices in the same part have the same degree). It is a longstanding conjecture of Hansen~\cite{H} (see also~\cite{JT,STSF}) that such graphs are, in fact, interval colourable.

\begin{conjecture}[\cite{H}]
Every biregular graph is interval colourable.
\end{conjecture}

We pose a weaker conjecture.

\begin{conjecture}\label{conj:bireg}
There exists an absolute constant $C$ such that $\theta(G)\leq C$ for every biregular graph $G$.
\end{conjecture}

A proof of Conjecture~\ref{conj:bireg} might pave the way to an improved bound on $\theta(n)$ via decompositions into forests and biregular graphs. We ask whether this strategy has the potential to prove Conjecture~\ref{conj:o1}.

 \begin{question}
Can every $n$-vertex graph be edge-decomposed into $n^{o(1)}$ forests and biregular graphs?
\end{question}

B\'ela Bollob\'as~\cite{BB} has remarked to us that the notion of an interval colouring could be generalised in the following way: given a constant $\alpha \geq 1$, we say that a graph $G = (V, E)$ is \emph{$\alpha$-interval colourable} if there exists a proper edge colouring $c : E \to \mathbb{Z}$ such that for every vertex $x$, the set $\{c(xy) : xy \in E\}$ of colours incident to $x$ is contained in an interval of $\mathbb{Z}$ of size at most $\alpha d(x)$. Defining $\theta_\alpha(G)$ in the obvious way, one only needs to change the proof of \Cref{thm:lb} very slightly to obtain the bound $\theta_\alpha(G)\geq c_\alpha (\log n)^{1-o(1)}$ for some constant $c_\alpha>0$ depending only on $\alpha$. It would be interesting, however, to see if one can prove a stronger upper bound than \Cref{thm:ub} in this setting.

In regard to the section on planar graphs, we pose the following problem of determining all extremal examples for \Cref{thm:3/2result}.

\begin{problem}
    Characterise the $n$-vertex planar graphs $G$ for which $t(G)=\floor{3n/2}-2$.
\end{problem}

\nopagebreak

\section*{Acknowledgements} The first author thanks the Mathematical Institute at the University of Oxford for their hospitality. The third and fourth authors would like to thank their PhD supervisor B\'ela Bollob\'as for his support.

\end{document}